\documentclass[a4paper,10pt]{amsart}

\usepackage{amsmath,amssymb}
\usepackage{amsthm}
\usepackage{enumitem}
\usepackage{mathtools}
\usepackage[all]{xy}

\theoremstyle{plain}
\newtheorem{theorem}{Theorem}[section]
\newtheorem{proposition}[theorem]{Proposition}
\newtheorem{lemma}[theorem]{Lemma}
\newtheorem{corollary}[theorem]{Corollary}

\theoremstyle{definition}
\newtheorem{definition}[theorem]{Definition}
\newtheorem*{acknowledgements}{Acknowledgements}

\newtheorem{example}[theorem]{Example}

\theoremstyle{remark}
\newtheorem{remark}[theorem]{Remark}

\newtheorem{convention}[theorem]{Convention}

\newtheorem{question}[theorem]{Question}

\newtheorem{step}{Step}

\numberwithin{equation}{theorem}

\DeclareMathOperator{\NE}{NE}
\DeclareMathOperator{\Proj}{Proj}
\DeclareMathOperator{\pr}{pr}

\DeclareMathOperator{\Gr}{Gr}
\DeclareMathOperator{\rank}{rank}

\DeclareMathOperator{\codim}{codim}

\DeclareMathOperator{\id}{id}
\DeclareMathOperator{\Chow}{Chow}
\DeclareMathOperator{\Fl}{Fl}

\DeclareMathOperator{\SFl}{SFl}
\DeclareMathOperator{\OG}{OG}
\DeclareMathOperator{\SG}{SG}
\DeclareMathOperator{\im}{Im}

\newcommand{\cNE}{{\overline{\NE}}}

\newcommand{\nequiv}{\equiv _\mathrm{num}}

\newcommand{\tX}{\widetilde X}

\newcommand{\te}{\widetilde e}

\newcommand{\tF}{\widetilde F}

\newcommand{\tS}{\widetilde{S}}
\newcommand{\tU}{\widetilde{U}}
\newcommand{\tpi}{\widetilde{\pi}}
\newcommand{\blank}{{-}}
\newcommand{\dto}{\dashrightarrow}

\newcommand{\sE}{\mathcal{E}}
\newcommand{\sF}{\mathcal{F}}
\newcommand{\sG}{\mathcal{G}}

\newcommand{\sC}{\mathcal{C}}

\newcommand{\sN}{\mathcal{N}}

\newcommand{\cO}{\mathcal{O}}

\newcommand{\bC}{\mathbf{C}}

\newcommand{\bP}{\mathbf{P}}
\newcommand{\bQ}{\mathbf{Q}}
\newcommand{\bR}{\mathbf{R}}
\newcommand{\bZ}{\mathbf{Z}}
\newcommand{\bO}{\mathbf{O}}

\newcommand{\dPA}{
{\SelectTips{}{12}
\objectmargin={0pt}
\objectheight={20pt}
\objectwidth={5pt}
\xygraph{!{<0cm,0cm>;<0.7cm,0cm>:<0cm,0.7cm>::}
\bullet *=!D{1} -[r] \circ*=!D{2}-@{-}|{\cdots}[r]  \circ -[r] \circ *=!D{r-1}
}}
}
\newcommand{\dPC}{
{\SelectTips{}{12}
\objectmargin={0pt}
\objectheight={20pt}
\objectwidth={5pt}
\xygraph{!{<0cm,0cm>;<0.7cm,0cm>:<0cm,0.7cm>::}
\bullet*=!D{1} -[r] \circ*=!D{2}-@{-}|{\cdots}[r]  \circ -@2{-}|@{<}[r] \circ*=!D{r/2}
}}
}
\newcommand{\dAA}{
{\SelectTips{}{12}
\objectmargin={0pt}
\objectheight={20pt}
\objectwidth={5pt}
\xygraph{!{<0cm,0cm>;<0.7cm,0cm>:<0cm,0.7cm>::}
\bullet*=!D{1} -[r] \circ*=!D{2}-@{-}|{\cdots}[r]  \circ-[r] \circ*=!D{r-1} -@0{-}[r] \bullet*=!D{1} -[r] \circ*=!D{2}-|{\cdots}[r] \circ-[r] \circ*=!D{r-1} 
}}
}

\newcommand{\dAM}{
{\SelectTips{}{12}
\objectmargin={0pt}
\objectheight={20pt}
\objectwidth={5pt}
\xygraph{!{<0cm,0cm>;<0.7cm,0cm>:<0cm,0.7cm>::}
\bullet*=!D{1} -[r] \circ*=!D{2}-|{\cdots}[r] \circ -[r] \bullet*=!D{r}
}}
}

\newcommand{\dAG}{
{\SelectTips{}{12}
\objectmargin={0pt}
\objectheight={20pt}
\objectwidth={5pt}
\xygraph{!{<0cm,0cm>;<0.7cm,0cm>:<0cm,0.7cm>::}
\circ*=!D{1}-|{\cdots}[r]  \circ -[r] \bullet*=!D{r-1}-[r] \bullet*=!D{r} -[r] \circ -|{\cdots}[r]  \circ*=!D{2r-2}
}}
}

\newcommand{\dC}{
{\SelectTips{}{12}
\objectmargin={0pt}
\objectheight={30pt}
\objectwidth={5pt}
\xygraph{!{<0cm,0cm>;<0.7cm,0cm>:<0cm,0.7cm>::}
\circ*=!D{1}-|{\cdots}[r] \circ -[r] \bullet*=!D{r-1}-[r] \bullet*=!D{r} -[r] \circ -|{\cdots}[r]  \circ-@2{-}|(.6)@{<}[r] \circ*=!D{\dfrac{3r}{2}-1}
}}
}

\newcommand{\dD}{
{\SelectTips{}{12}
\objectmargin={0pt}
\objectheight={20pt}
\objectwidth={5pt}
\xygraph{!{<0cm,0cm>;<0.7cm,0cm>:<0cm,0.7cm>::}
\circ*=!D{1}-|{\cdots}[r] \circ -[r] \circ(-[]!{+(.7,.5)}\bullet*=!D{r-1}) -[]!{+(.7,-.5)} \bullet*=!D{r}
}}
}

\newcommand{\dF}{
{\SelectTips{}{12}
\objectmargin={0pt}
\objectheight={20pt}
\objectwidth={5pt}
\xygraph{!{<0cm,0cm>;<0.7cm,0cm>:<0cm,0.7cm>::}
\circ-[r] \bullet -@2{-}|(.6)@{>}[r] \bullet-[r] \circ 
}}
}

\newcommand{\dG}{
{\SelectTips{}{12}
\objectmargin={0pt}
\objectheight={20pt}
\objectwidth={5pt}
\xygraph{!{<0cm,0cm>;<0.7cm,0cm>:<0cm,0.7cm>::}
\bullet-@3{-}|(.6)@{>}[r] \bullet 
}}
}

\AtBeginDocument{%
\def\MR#1{}
}

\setenumerate{label=(\arabic*),
font=\normalfont
}

\title{Mukai pairs and simple $K$-equivalence}
\author[A. KANEMITSU]{Akihiro KANEMITSU}
\date{\today}
\address{Research Institute for Mathematical Sciences,
Kyoto University, Kyoto 606-8502, Japan}
\email{kanemitu@kurims.kyoto-u.ac.jp}
\thanks{The author is a JSPS Research Fellow and he is supported by the Grant-in-Aid for JSPS fellows (JSPS KAKENHI Grant Number 18J00681).}

\begin{document}

\begin{abstract}
A $K$-equivalent map between two smooth projective varieties is called simple if the map is resolved in both sides by single smooth blow-ups.
In this paper, we will provide a structure theorem of simple $K$-equivalent maps, which reduces the study of such maps to that of special Fano manifolds.
As applications of the structure theorem, we provide examples of simple $K$-equivalent maps, and classify such maps in several cases, including the case of dimension at most $8$.
\end{abstract}

\maketitle

\section*{Introduction}
A \emph{$K$-equivalent map} between two smooth projective varieties $X_1$ and $X_2$ is, by definition, a birational map $\chi \colon X_1 \dto X_2$ that  admits a resolution of indeterminacy
\[
\xymatrix{
& \tX \ar[ld]_-{f_1} \ar[rd]^-{f_2}& \\
X_1 \ar@{-->}[rr]^{\chi}& & X_2
}
\]
by a smooth projective variety $\tX$ with the condition $f_1^{*}K_{X_1} = f_2^*K_{X_2}$.
Such birational maps appear in several important situations of birational geometry of algebraic varieties;
for example, flops are $K$-equivalent birational maps, and any two birational minimal varieties are $K$-equivalent.
Also, it is checked or conjectured that $K$-equivalence preserves many invariants of algebraic varieties;
for example, Kawamata's $DK$-hypothesis predicts that $K$-equivalence of two algebraic varieties implies their $D$-equivalence, i.e.\ their derived categories of coherent sheaves are equivalent \cite{Kaw02}.

In this paper, we will focus on a class of $K$-equivalent birational maps, called simple $K$-equivalent maps.
A $K$-equivalent map is called \emph{simple}, if we can choose a resolution as above such that $f_i$ are smooth blow-ups \cite{Li18}.
At a first glance, the assumption in this definition seems to be too strong.
However, this class is very interesting because it includes some important birational maps such as \emph{standard flops} and \emph{Mukai flops}, and it provides nice examples for testing several conjectures on $K$-equivalent birational maps.
For example, $D$-equivalence for standard flops and Mukai flops are proved in \cite{BO95,Kaw02,Nam03}.
Also, in \cite{Seg16}, it is proved that (in a local setting) a simple $K$-equivalent map in dimension $5$, called \emph{Abuaf's flop}, induces $D$-equivalence (cf.\ \cite{Har17}).
A similar statement for a $7$-dimensional flop is also obtained by Ueda \cite{Ued18}.

Based on the above interesting phenomena, it is natural to wonder further examples of simple $K$-equivalent birational maps, and try to classify these birational maps.
Such an attempt is started by \cite{Li18}, and it is proved that simple $K$-equivalent maps in dimension at most $5$ are only three types; standard flops, Mukai flops and Abuaf's flop.
Also it is desirable to have a nice structure theorem for simple $K$-equivalent maps.
In the present paper, we go further in this direction.
More precisely, the purposes of this paper are
\begin{enumerate}
 \item to give a structure theorem of simple $K$-equivalent maps, which relates such maps to a special kind of Fano manifolds, which we call \emph{roofs};
 \item to provide applications of the structure theorem. More precisely, we provide examples of $K$-equivalent birational maps and classify such maps in several cases.
\end{enumerate}

\subsection{Results}
In order to state the structure theorem, we introduce some notions:
\begin{definition}[Mukai pairs and roofs]\label{def:mukai}
\hfill
\begin{enumerate}
\item \cite{Muk88} A \emph{Mukai pair} $(V,\sE)$ of dimension $n$ and rank $r$ is a pair of a Fano $n$-fold $V$ and an ample vector bundle $\sE$ of rank $r$ which satisfies $c_1(V) = c_1 (\sE)$.
\item A Mukai pair of rank $r$ is called \emph{simple} if the Picard number of $V$ is one, and the projectivization $\bP(\sE)$ admits another $\bP^{r-1}$-bundle structure.
\item A \emph{roof of $\bP^{r-1}$-bundles} is a Fano manifold $W$ that is isomorphic to the projectivization of a simple Mukai pair with rank $r$.
\end{enumerate}
\end{definition}
Later we will see that a Fano manifold $W$ is a roof of $\bP^{r-1}$-bundles if and only if the following three conditions are satisfied (see Proposition~\ref{prop:roof} for several characterizations of roofs):
\begin{enumerate}
 \item The Picard number of $W$ is two.
 \item $W$ admits two (different) $\bP^{r-1}$-bundle structures.
 \item The index of $W$ is $r$, i.e.\ $-K_W = r H_W$ for some Cartier divisor $H_W$.
\end{enumerate}

Now we can state the structure theorem of simple $K$-equivalent maps.
Let $\chi \colon X_1 \dto X_2$ be a simple $K$-equivalent map between two smooth projective varieties, and let the following diagram
\begin{equation}\label{diagram:simple}
 \xymatrix{
 & & E \ar@{^{(}->}[d] \ar[lldd]_-{g_1} \ar[rrdd]^-{g_2} & & \\
 & & \tX \ar[ld]_-{f_1} \ar[rd]^-{f_2} & & \\
 Y_1 \ar@{^{(}->}[r] & X_1 \ar@{-->}[rr]^{\chi} & & X_2 & Y_2\ar@{_{(}->}[l]
}
\end{equation}
be its resolution by two smooth blow-ups along $Y_1$ and $Y_2$.
We always assume that $\chi$ is not an isomorphism.
Note that by \cite[Lemma~2.1]{Li18} (see Lemma~\ref{prop:excdiv}) the exceptional divisors of $f_1$ and $f_2$ coincide, which we denoted by $E$, and that $\dim Y_1 = \dim Y_2$.
In the following, we will denote by $r$ the codimension of $Y_i$ in $X_i$ and by $\sC_{Y_i/X_i}$ the conormal bundle of $Y_i$ in $X_i$.
Thus $\dim X_1 = \dim X_2 =\dim Y_1+r = \dim Y_2 + r$ and $E \simeq \bP (\sC_{Y_i/X_i})$.

\begin{theorem}[Structure theorem]\label{thm:simple}
Let $\chi \colon X_1 \dto X_2$ be a simple $K$-equivalent map between two smooth projective varieties and let the notation be as above.
Then there exist a smooth projective manifold $M$ and the following commutative diagram
\begin{equation}\label{diagram:str}
 \xymatrix{
 & & E \ar@{^{(}->}[d] \ar[lldd]_-{g_1} \ar[rrdd]^-{g_2} & & \\
 & & \tX \ar[ld]_-{f_1} \ar[rd]^-{f_2} & &\\
 Y_1 \ar@{^{(}->}[r] \ar[rrdd]_-{h_1} & X_1 \ar@{-->}[rr]^{\chi} & & X_2 & Y_2\ar@{_{(}->}[l] \ar[lldd]^-{h_2} \\
 & & & & \\
 & & M, & &
}
\end{equation}
which satisfy the following conditions:
\begin{enumerate}
 \item $h_i$ ($i=1$, $2$) are smooth extremal contractions.
 \item For each $h_i$-fiber $F_i$, the pair $(F_i, \sC_{Y_i/X_i}|_{F_i})$ is a simple Mukai pair.
 \item Each $\psi$-fiber is a roof of $\bP^{r-1}$-bundles, where $\psi \coloneqq h_i \circ g_i$.
\end{enumerate}
\end{theorem}

Roughly speaking, the theorem says that a simple $K$-equivalent map is a family of more simpler maps induced from simple Mukai pairs. This theorem is proved in Section~\ref{sect:BB} after the preparation in Section~\ref{sect:twoPB}.

Conversely, in Section~\ref{sect:const}, we will explain how we can construct simple $K$-equivalent maps from simple Mukai pairs.
More generally, we will construct a simple $K$-equivalent map $X \dto X^{+}$ to a complex manifold $X^+$ (which may not be projective in general) from the following given data:
\begin{enumerate}
 \item $X$ is a smooth projective variety, and $Y \subset X$ is a smooth closed subvariety of $X$.
 \item $Y$ admits a smooth extremal contraction $h \colon Y \to M$.
 \item Each $h$-fiber $F$ is a Fano manifold with Picard number one, and the pair $(F,\sC_{Y/X}|_F)$ is a simple Mukai pair.
\end{enumerate}
This construction follows \cite[Section 3]{Muk84}.
Also we will construct the local model of simple $K$-equivalent map from a simple Mukai pair (cf.\ \cite[Section~1]{Nam03}).
Therefore, the study of simple $K$-equivalence is (locally) equivalent to that of simple Mukai pairs.

Then, in Section~\ref{sect:ex}, we will construct several simple $K$-equivalent maps by using the inverse construction.
More precisely, we will construct eight types of such maps, which we will denote by type $A_{r-1} \times A_{r-1}$, $A_r^M$, $A_{2r-2}^G$, $C_{\frac{3r}{2}-1}$ ($r$ even), $D_r$, $F_4$ ($r=3$), $G_2$ ($r=2$) and $G_2^{\dagger}$ ($r=3$) respectively.
All of these examples are deeply related to semi-simple algebraic groups.
Indeed, the corresponding roofs are all homogeneous, with one exception of type $G_2^{\dagger}$.
Also, this exception, the roof of type $G_2^{\dagger}$ has its origin to the geometry of the Cayley octonions and admits the action of the exceptional group of type $G_2$.
In that section, we also collect partial classification results of roofs, which are consequences of the classification of Mukai pairs with large ranks
\cite{Fuj92,Pet90,Pet91,YZ90,Wis89b,PSW92b,NO07,Kan17Mukai,Kan18Mukai} (cf.\ \cite{Occ05})
and the classification of Fano manifolds with Picard rank two whose extremal contractions are $\bP^1$-bundles \cite{MOS14a,Wat14b}.
Then, by combining these classification results with the structure theorem, we will prove the following theorem:
\begin{theorem}[= Corollary~\ref{cor:class}]\label{thm:dim8}
Let $\chi \colon X_1 \dto X_2$ be a simple $K$-equivalent map in codimension $r$, and let the notation be as in Theorem~\ref{thm:simple}.
Assume one of the following conditions:
\begin{enumerate}
 \item $r \geq \dim Y_i - \dim M -2$.
 \item $r = 2$.
 \item $\dim X_i \leq 8$.
\end{enumerate}
Then $\chi$ is one of the above eight types.
\end{theorem}

\begin{remark}
 As mentioned above, we will construct eight examples of simple $K$-equivalent maps in Section~\ref{sect:ex}.
Some of these examples are classical or well-known: $K$-equivalent maps of type $A_{r-1} \times A_{r-1}$ and $A_r^M$ are standard flops and Mukai flops respectively.
Abuaf's flop in \cite{Seg16} is of type $C_2$, and the $7$-dimensional flop discussed in \cite{Ued18} is of type $G_2$.
Also, though this author could not find them in the literature, some of the other examples seem to be known to the experts;
for example, in response to the earlier version of this paper, Doctor Duo Li informed this author that he also realized the idea to construct simple $K$-equivalence from homogeneous varieties, and Hua-Zhong Ke had an idea to relativize Abuaf's flop.
\end{remark}

In the last section (=Section~\ref{sect:symp}), as an application of Theorem~\ref{thm:dim8}, we will provide an answer to a question of Daniel Huybrechts on simple $K$-equivalence on symplectic varieties.

\begin{convention}
We will work over the complex number field $\bC$.
A \emph{smooth $\bP^{r-1}$-fibration} means a smooth projective morphism whose fibers are projective spaces $\bP^{r-1}$, while a \emph{$\bP^{r-1}$-bundle} means the projection of a projectivized vector bundle.
For a vector bundle $\sE$ on a variety $V$, we will denote by $\bP(\sE)$ the projectivization $\Proj (S (\sE))$ in the sense of Grothendieck.
\end{convention}

\begin{acknowledgements}
The author wishes to express his gratitude to Professor Yoichi Miyaoka for suggesting me the relation between Mukai pairs and flops, and to Professor Shigeru Mukai for drawing the authors's attention to \cite{Li18} and for his insightful comments and discussions.
The author is also grateful to Professor Yuki Hirano for his helpful comments and telling me about the paper \cite{Ued18}, and to Doctor Sho Ejiri for helpful discussions on a variant of Theorem~\ref{thm:quotient}.
He is also grateful to Doctor Wahei Hara for his helpful comments.
The author wishes to thank Doctor Duo Li  for his helpful comments and discussions on the earlier version of this paper, and for permitting this author to include Theorem~\ref{thm:symp}, which is obtained by the discussion with him.

A part of this paper was written during the author's stay at the University of Trento.
He is grateful to this institution for its hospitality and financial support.
\end{acknowledgements}

\section{Preliminaries}
\subsection{Fundamental properties of simple $K$-equivalence}
Let $\chi \colon X_1 \dto X_2$ be a simple $K$-equivalent map, and 
\[
\xymatrix{
 &E_1 \ar@{^{(}->}[dr] \ar[ldd]_-{g_1} & &E_2 \ar@{_{(}->}[dl] \ar[rdd]^-{g_2} & \\
 & & \tX \ar[ld]_-{f_1} \ar[rd]^-{f_2} & & \\
 Y_1 \ar@{^{(}->}[r] & X_1 \ar@{-->}[rr]^{\chi} & & X_2 & Y_2 \ar@{_{(}->}[l]
}
\]
be its resolution of indeterminacy by smooth blow-ups along $Y_i \subset X_i$.
Here $E_1$ and $E_2$ are the exceptional divisors.
In what follows, we will tacitly assume the condition $f_1^{*}K_{X_1} = f_2^*K_{X_2}$.

\begin{proposition}[Exceptional divisors {\cite[Lemma~2.1]{Li18}}]\label{prop:excdiv}
Let the notaion be as above.
Then $E_1 =E_2$ and $\codim_{X_1} Y_1 = \codim_{X_2} Y_2$.
\end{proposition}

\begin{proof}
Set $r_i \coloneqq \codim_{X_i} Y_i$.
Then
\[
 K_{\tX} \sim f_i^*K_{X_i}+(r_i-1)E_i.
\]
Thus the condition $f_1^{*}K_{X_1} = f_2^*K_{X_2}$ together with the above equality yields our assertions (note that $E_i$ are exceptional divisors).
\end{proof}

Thus we have the following diagram as in \eqref{diagram:simple}:

\begin{equation}
\xymatrix{
 & & E \coloneqq E_1=E_2 \ar@{^{(}->}[d] \ar[lldd]_-{g_1} \ar[rrdd]^-{g_2} & & \\
 & & \tX \ar[ld]_-{f_1} \ar[rd]^-{f_2} & & \\
 Y_1 \ar@{^{(}->}[r] & X_1 \ar@{-->}[rr]^{\chi} & & X_2 & Y_2 \ar@{_{(}->}[l] \\
}
\end{equation}

\begin{definition}[Codimension and exceptional divisor]\label{def:codim}
Let $\chi \colon X_1 \dto X_2$ be a simple $K$-equivalent map, and the notation as above.
Then its \emph{codimension} $r$ is defined as $\codim_{X_1} Y_1 = \codim_{X_2} Y_2$.
Its \emph{exceptional divisor} $E$ is the exceptional divisor of $f_i$.
\end{definition}

As a corollary of Proposition~\ref{prop:excdiv}, we have the following:

\begin{corollary}[Two projective bundle structures]\label{cor:twoPB}
 Let $E$ be the exceptional divisor of a simple $K$-equivalent map in codimension $r$.
 Then $E$ admits two $\bP^{r-1}$-bundle structures $g_1$ and $g_2$.
\end{corollary}

Set $\cO_E(1) \coloneqq \cO(-E)|_E$ and $\cO_E(m)\coloneqq \cO_E(1)^{\otimes{m}}$.
Since $f_i$ is a smooth blow-up, we have $E \simeq \bP_{Y_i}(\sC_{Y_i/X_i})$, where $\sC_{Y_i/X_i}$ is the conormal bundle of $Y_i$ in $X_i$.
Moreover the line bundle $\cO_E (1)$ gives the relative tautological bundle of this projectivization.
By an easy calculation, we have
\[
r \cO_E(1) +g_i^*(-K_{X_i}|_{Y_i}) = -K_{E} = r\cO_E(1) +g_i^*(-K_{Y_i}-\det (\sC_{Y_i/X_i})).
\]
Thus, we have the following:

\begin{proposition}\label{prop:canonicalbundle}
\hfill
\begin{enumerate}
 \item $\cO_E(1)$ is the relative tautological divisor of the projective bundles $g_1$ and $g_2$.
 \item $c_1 (\sC_{Y_i/X_i}) = -K_{Y_i} + K_{X_i}|_{Y_i}$.
\end{enumerate}
\end{proposition}

\subsection{Characterization of roofs}
Let $(V,\sE)$ be a simple Mukai pair with rank $r$, and $W$ the roof $\bP_V(\sE)$.
We will denote by $\xi_{\sE}$ the relative tautological divisor of this projectivization $\bP_V(\sE)$.
Then $W$ admits another $\bP^{r-1}$-bundle structure $g^+ \colon W \to V^+$.
Since $-K_W = r \xi_{\sE}$, the divisor $\xi_{\sE}$ restricts to $\cO(1)$ on each $g^+$-fiber $\bP^{r-1}$.
Thus $g^+ \colon W \to V^+$ is given by the projectivization of the vector bundle $\sE^+ \coloneqq (g^+)_*\cO(\xi_\sE)$.
By \cite[Proposition~3.3]{NO07}, the pair $(V^+, \sE^+)$ is also a Mukai pair.
Thus the situation is symmetric in $(V,\sE)$ and $(V^+,\sE^+)$.
The following proposition gives easy, but useful, characterizations of roofs of $\bP^{r-1}$-bundles:

\begin{proposition}[Roofs]\label{prop:roof}
 Let $W$ be a smooth projective Fano manifold of Picard number two.
 Assume that every extremal contraction $W \to V_i$ ($i=1$, $2$) is a smooth $\bP^{r-1}$-fibration.
 Then the following are equivalent:
\begin{enumerate}
 \item \label{prop:roof1} $W$ is a roof of $\bP^{r-1}$-bundles.
 \item \label{prop:roof2} The index of $W$ is $r$.
 \item \label{prop:roof3} There exists a divisor $D$ on $W$ such that $\cO_W(D)$ restricts to $\cO(1)$  on all fibers $\bP^{r-1}$ of both extremal contractions.
\end{enumerate}
\end{proposition}

\begin{proof}
 We have already seen \ref{prop:roof1} $\implies$ \ref{prop:roof2} and \ref{prop:roof1} $\implies$ \ref{prop:roof3}.
 By \cite[Proposition~3.3]{NO07}, \ref{prop:roof2} implies \ref{prop:roof1}.
 Also, by adjunction, \ref{prop:roof2} $\implies$ \ref{prop:roof3}.
 
Assume \ref{prop:roof3}. Then $-K_W \nequiv rD$, since they coincide on each $g_i$-fiber and $N_1(W)$ is spanned by $g_i$-fibers.
 Since numerical equivalence and linear equivalence coincide on Fano manifolds, \ref{prop:roof2} holds.
\end{proof}

\begin{remark}
Let $W$ be a Fano variety as in the assumption of Proposition~\ref{prop:roof}.
To the best of the author's knowledge, there are no examples $W$ which do not satisfy these equivalent conditions \ref{prop:roof1}--\ref{prop:roof3}.
\end{remark}

\section{Manifolds with two projective bundle structures}\label{sect:twoPB}
Let $\chi \colon X_1 \dto X_2$ be a simple $K$-equivalent map in codimension $r$.
Then, by Corollary~\ref{cor:twoPB}, the exceptional divisor $E$ admits two $\bP^{r-1}$-bundle structures.
In this section we will study the structure of its Kleiman-Mori cone $\cNE(E)$.
Roughly speaking, the results of this section show that, if a projective manifold admits two projective bundle structures, then the corresponding rays $R_1$ and $R_2$ span a two-dimensional extremal face in its Kleiman-Mori cone, and the contraction of this face makes $E$ a family of  Fano manifolds with two projective bundle structures. 

\subsection{}

We start with a more general situation as follows:
Let $X$ be a normal projective variety.
A \emph{basic diagram} on $X$ is a diagram of the following form:
\[
\xymatrix{
 & U \ar[ld]_-{\pi} \ar[rd]^-{e} & \\
 S & & X,
}
\]
where $U$ and $S$ are normal projective varieties.
In what follows, we will assume $\pi_*\cO_U = \cO _S$ for simplicity, and hence all the $\pi$-fibers are connected.
Then the $S$-equivalent relation on $X$ is defined as follows: two points $x_1$ and $x_2$ are said to be \emph{$S$-equivalent} if these two points are contained in a connected chain of $e$-images of $\pi$-fibers, i.e.\ there are (finite) points $s_j \in S$ such that $x_i \in \bigcup e(\pi^{-1}(s_j))$ and $\bigcup e(\pi^{-1}(s_j))$ is connected.
In this situation, it is known that there is a rational map $X \dto Y $ which gives, not on the whole variety $X$ but on an open subset of $X$, the quotient map for this $S$-equivalent relation \cite{Cam81,KMM92a} (see also \cite[Chapter 5]{Deb01}, \cite[Chapter IV]{Kol96}).
More precisely, we have:

\begin{theorem}
There exist a non-empty open subset $X^0 \subset X$ and a projective morphism $q: X^0 \to Y^0$ such that each $q$-fiber is an $S$-equivalent class. 
\end{theorem}

The above map $q$ is called the \emph{$S$-equivalent quotient map}.
For accounts of this topic, our basic references are \cite[Chapter 5]{Deb01}, \cite[Chapter IV]{Kol96}.

In general, the quotient map is not defined on the whole variety $X$.
Thus it is natural to ask when the quotient map is defined on the whole variety $X$.
For example, in \cite[Section 2]{Kan17}, it is proved that, if $\pi$ and $e$ are smooth $\bP^1$-fibrations and all varieties are smooth, then the quotient map is actually a smooth morphism defined on the whole variety $X$.
The following theorems \ref{thm:quotient} and \ref{thm:smooth} generalize this theorem.

\begin{theorem}[Quotient map]\label{thm:quotient}
Let $(S \xleftarrow{\pi} U \xrightarrow{e} X)$ be a basic diagram on $X$.
Assume that $e$ and $\pi$ are equidimensional with irreducible fibers.
Then there is a projective morphism $q \colon X \to M$ onto a projective normal variety $M$ whose fibers are the $S$-equivalent classes.
Moreover the map $q$ is equidimensional with irreducible fibers.
\end{theorem}

\begin{theorem}[Smoothness of contraction {\cite[Theorem~5.2]{DPS94}, \cite[Theorem~4.4]{SW04}}]\label{thm:smooth}
Let $(S \xleftarrow{\pi} U \xrightarrow{e} X)$ be a basic diagram on $X$ as in Theorem~\ref{thm:quotient}.
Assume moreover the following three conditions:
\begin{enumerate}
\item $X$ is smooth.
\item $\pi$ is a smooth $\bP^1$-fibration.
\item $e^*T_X$ is $\pi$-nef. 
\end{enumerate}
Then the quotient morphism $q$ is smooth.
\end{theorem}

The proof of Theorem~\ref{thm:quotient} relies on the theory of algebraic cycles and Chow varieties.
For a detailed account of families of algebraic cycles and Chow varieties, we refer the reader to \cite[Chapter I]{Kol96}.

\begin{proof}[Proof of Theorem~\ref{thm:quotient}]
Consider the following diagram obtained by taking products with $X$:
\[
\xymatrix{
 & U \times X \ar[ld]_-{\Pi \coloneqq \pi\times \id} \ar[rd]^-{E \coloneqq e \times \id} & \\
 S \times X & & X \times X.
}
\]

Let $V_0$ be the diagonal in $X \times X$ and define inductively $V_{i+1}$ as
\[
E(\Pi^{-1}(\Pi(E^{-1}(V_i))))
\]
with its reduced strucure.
We will consider  $V_i$ as a scheme over $X$ via the second projection $\pr_2$, and denote by $V_i(x)$ the fiber $V_i \cap \pr_2^{-1}(x)$.
By the construction, $V_i(x)$ is the set of points that can be connected to $x$ by an $S$-chain of length $i$.

\begin{step}
Here we will prove, by induction on $i$, that $\pr_2 \colon V_i \to X$ is a well-defined family of irreducible algebraic cycles (see \cite[Chapter I, Section 3]{Kol96} for the definition and properties).
Trivially $\pr_2 \colon V_0 \to X$ is a well-defined family of irreducible algebraic cycles.
Assume that $\pr_2 \colon V_i \to X$ is a well-defined family of irreducible algebraic cycles.
Since $E \colon U\times X \to X \times X$ is equidimensional with irreducible fibers,  so is the map $\pr_2 \colon E^{-1}(V_i) \to X$.
Thus $\pr_2 \colon E^{-1}(V_i) \to X$ is a well-defined family of irreducible algebraic cycles by \cite[Chapter I, Theorem~3.17]{Kol96}.
Next consider $\pr_2 \colon \Pi(E^{-1}(V_i)) \to X$.
Take ample divisors $H_S$ on $S$ and $H_U$ on $U$, and denote by $d$ the relative dimension of $\pr_2 \colon E^{-1}(V_i) \to X$.
Then, for $t \in \bR$,  the number
\[
(\pi^{*}H_S+tH_U)^d\cdot(E^{-1}(V_i)\cap \pr_2 ^{-1}(x))
\]
is independent of $x \in X$ by \cite[Chapter IV, Prop~2.10]{Kol96}.
Thus the morphism $\pr_2 \colon \Pi(E^{-1}(V_i)) \to X$ is equidimensional with irreducible fibers, and hence $\pr_2 \colon \Pi(E^{-1}(V_i)) \to X$ is again a well-defined family of irreducible algebraic cycles.
By iterating this procedure, we see that $\pr_2 \colon V_{i+1} \to X$ is a well-defined family of irreducible algebraic cycles.
\end{step}

\begin{step}
In this step, we construct the quotient morphism.
Since $\pr_2 \colon V_i \to X$ is a well-defined family of irreducible algebraic cycles, the total space $V_i$ is irreducible.
Hence there exists an integer $k$ such that
\[
V_0 \subsetneq V_1 \subsetneq \cdots \subsetneq V_k = V_{k+1} = \cdots \eqqcolon V_\infty.
\]
By definition of $V_i$, the fiber $V_{\infty}(x) \coloneqq V_{\infty}\cap \pr _2^{-1}(x)$ is the $S$-equivalent class of $x \in X$.
Since $\pr_2 \colon V_{\infty} \to X$ is a well-defined family of irreducible algebraic cycles, we have a morphism $\overline q \colon X \to \Chow (X)$ by the universal property of Chow varieties.
Let $X \xrightarrow{q} M \to \Chow(X)$ be the Stein factorization of $\overline q$.
Then, since $X$ and $\Chow(X)$ are projective, the morphism $q$ and $M$ are projective.
By construction each fiber of the morphism $q$ is an $S$-equivalent class.
This completes the proof.
\end{step}
\end{proof}

\begin{proof}[Proof of Theorem~\ref{thm:smooth}]
The proof is essentially the same as the proofs of \cite[Theorem~5.2]{DPS94} and \cite[Theorem~4.4]{SW04}.
Here we will only provide the outline of the proof, based on the proof of \cite[Theorem~4.4]{SW04}.
Note that $q$ is equidimensional with irreducible fibers.
Then, by arguing as in the proof of \cite[Lemma 4.12]{SW04}, we know that the following are satisfied:
\begin{enumerate}
 \item Every $q$-fiber with its reduced structure, denoted by $F$, is a smooth Fano manifold.
 \item By restricting the basic diagram $(S \xleftarrow{\pi} U \xrightarrow{e} X)$, we have a basic diagram $(S_{F} \xleftarrow{\pi_{F}} U_{F} \xrightarrow{e_{F}} F)$ such that $\pi_F$ is a smooth $\bP^1$-fibration, $F$ is chain-connected with respect to these families and the bundle $e_{F}^*(\sN_{F/X})$ is trivial on each $\pi_F$-fibers.
\end{enumerate}
Then, Lemma~\ref{lem:trivial} below shows that $\sN_{F/X}$ is trivial.
Hence the contraction $q$ is smooth by \cite[Lemma~4.13]{SW04}.
\end{proof}

\begin{lemma}[{\cite[Proposition~1.2]{AW01}}]\label{lem:trivial}
Let $F$ be a smooth Fano variety and $\sE$ be a vector bundle on $F$.
Assume that there exists a basic diagram $(S \xleftarrow{\pi} U \xrightarrow{e} F)$ such that  $\pi$ is a smooth $\bP^1$-fibration, $F$ is chain-connected with respect to this family and $\sE$ is trivial on each $\pi$-fiber.
Then $\sE$ is trivial.
\end{lemma}

\begin{proof}
This lemma follows from a similar argument as in the proof of \cite[Proposition~1.2]{AW01}.

Consider the bundle $e^* \sE$.
Then, since  $e^* \sE$ is trivial on each $\pi$-fiber, the push-forward $\sF \coloneqq \pi_*e^* \sE$ is a vector bundle on $S$ and we have an isomorphism $e^* \sE \simeq \pi^*\sF$.
Thus we have the following commutative diagram by taking projectivizations:
\[
\xymatrix{
\tS \coloneqq \bP(\sF) \ar[d] & \tU \coloneqq \bP(e^* \sE) \ar[d] \ar[r]^-{\te} \ar[l]_-{\tpi} & \bP(\sE) \ar[d] \\
 S & U \ar[r]^-{e} \ar[l]_-{\pi} & F.
}
\]
Then, by considering $(\tS \xleftarrow{\tpi} \tU \xrightarrow{\te} \bP(\sE))$ as a basic diagram on $\bP(\sE)$, we have the $\tS$-equivalent relation on $\bP(\sE)$.

Consider an $\tS$-equivalent class $\tF$.
Then, by the assumption on $\sE$ and the fact that $F$ is chain-connected with respect to $S$,  we see that the map $\tF \to F$ is surjective.
On the other hand, by \cite[Chapter IV, Proposition~3.13.3]{Kol96}, the image of the map $N_1(\tF) \to N_1(\bP(\sE))$ is a one-dimensional vector space spanned by the class of $\tpi$-fibers.
Thus the map $\tF \to F$ is finite and surjective, and hence $\dim \tF =\dim F$.
Therefore the image $Q$ of the $\tS$-quotient map $\bP(\sE) \dto Q$  has dimension $\rank E -1$. Now the assertion follows from \cite[Lemma~4.1]{NO07}.
\end{proof}

\begin{corollary}[Two projective bundles]\label{cor:face}
Let $U$, $S_1$ and $S_2$ be smooth projective varieties and $p_i \colon U \to S_i$ be smooth $\bP^{r_i-1}$-fibrations ($i=1,2$).
Denote by $R_i$ the extremal ray of $p_i$.
Then $R_1$ and $R_2$ span a two dimensional extremal face in $\cNE(U)$.
Moreover its contraction is smooth and each fiber of the contraction is a Fano manifold with Picard number two.
\end{corollary}

\begin{proof}
By the assumption, we have the following diagram
\begin{equation}\label{diagram:twoPB}
\xymatrix{
& U \ar[ld]_-{p_1} \ar[rd]^-{p_2} & \\
S_1 & & S_2.
}
\end{equation}
Considering this diagram as a basic diagram on $S_2$ and applying Theorem~\ref{thm:quotient}, we have the quotient morphism $q_2\colon S_2 \to M$.
Then, by rigidity lemma (see for instance \cite[Lemma~1.15]{Deb01}), we have a morphism $q_1 \colon S_1 \to M$, which makes the following diagram commutative:
\begin{equation}
\xymatrix{
 & U \ar[ld]_-{p_1} \ar[rd]^-{p_2}& \\
 S_1\ar[rd]_-{q_1} & & S_2 \ar[ld]^-{q_2}\\
 &M.&
}
\end{equation}
By symmetry, the morphism $q_1\colon S_1 \to M$ is also the quotient map for $S_2$-equivalent relation on $S_1$ that is induced by the diagram \eqref{diagram:twoPB}.
Note that the relative Picard rank $\rho(S_i/M)$ is one by \cite[Chapter IV, Proposition~3.13.3]{Kol96}.
Thus the morphism $U \to M$ is a contraction of a two dimensional face in $\cNE (U)$, and hence we have the first assertion.
By considering the family of lines in $p_1$-fibers, we have the following diagram:
\[
 \xymatrix{
 \widetilde U \ar[d]_-{\widetilde p_1} \ar[r]^-{e_1} & U \ar[d]_-{p_1} \ar[r]^-{p_2} & S_2 \\
 \widetilde S_1 \ar[r] & S_1, &
}
\]
where $\widetilde p_1 \colon \widetilde U \to \widetilde S_1$ is the universal family of lines in the $p_1$-fibers and $e_1\colon \widetilde U \to U $ is the evaluation map for this family.
Then, by considering $( \widetilde S_1 \xleftarrow{\widetilde p_1} \widetilde U \xrightarrow{p_2 \circ e_1} S_2)$ as a basic diagram on $S_2$, we have the $\widetilde S_1$-equivalent relation on $S_2$, which coincides with the $S_1$-equivalent relation on $S_2$.
Thus the map $q_2$ is the $\widetilde S_1$-quotient morphism.

Since $p_2 $ is smooth, we have the surjection $T_{ U} \to p_2 ^*T_{S_2}$.
Since $T_{ U}$ is $p_1$-nef, the bundle $p_2 ^*T_{S_2}$ is also $p_1$-nef.
Therefore, the bundle $(p_2 \circ e_2)^* T_{S_2}$ is $\widetilde p_1$-nef.
Hence, by Theorem~\ref{thm:smooth}, the contraction $q_2$ is smooth.
By symmetry, $q_1$ is also smooth.

Each fiber $F_2$ of $q_2$ is a smooth projective variety, and it is an $\widetilde S_1$-equivalent class.
Thus the Picard number of $F_2$ is one by \cite[Chapter IV, Proposition~3.13.3]{Kol96}.
Thus, for each fiber $F$ of  $U \to M$, the Picard number $\rho (F)$ is two.
This completes the proof.
\end{proof}

\section{Building blocks of simple $K$-equivalent maps}\label{sect:BB}
In this section, we completes the proof of Theorem~\ref{thm:simple}.
Let $\chi \colon X_1 \dto X_2$ be a simple $K$-equivalent map in codimension $r$ and consider the resolution of indeterminacy as in \eqref{diagram:simple}.
Then $g_i$ are $\bP^{r-1}$-bundles by Corollary~\ref{cor:twoPB}.
Thus, by applying Corollary~\ref{cor:face}, we have smooth extremal contractions $h_i \colon Y_i \to M$ with the following commutative diagram:
\[
\xymatrix{
 & &E \ar@{^{(}->}[d] \ar[lldd]_-{g_1} \ar[rrdd]^-{g_2} & & \\
 & & \tX \ar[ld]_-{f_1} \ar[rd]^-{f_2} & & \\
Y_1 \ar@{^{(}->}[r] \ar[rrdd]_-{h_1} & X_1 \ar@{-->}[rr]^{\chi} & & X_2 & Y_2 \ar@{_{(}->}[l] \ar[lldd]^-{h_2} \\
 & & & & \\
 & & M. & &
}
\]
We will denote by $\psi$ the composite $h_1 \circ g_1=  h_2 \circ g_2$.
Then, for each $m \in M$, the fiber $\psi^{-1}(m)$ is a Fano manifold with Picard number two whose extremal contractions are $\bP^{r-1}$-bundles:
\[
\xymatrix{
&\psi^{-1}(m) \ar[ld]_{g_1|_{\psi^{-1}(m)}} \ar[rd]^{g_2|_{\psi^{-1}(m)}}&\\
h_1^{-1}(m) & & h_2^{-1}(m).
}
\]
Note that each projective bundle structure is given by $\bP(\sC_{Y_i/X_i}|_{h_i^{-1}(m)})$.

The following lemma asserts that the canonical bundle of $X_i$ is trivial on each $h_i$-fiber, and hence the situation is very similar to the case of flops.

\begin{lemma}\label{lem:K=0}
For a point $m\in M$, we have
$K_{X_i}|_{h_i^{-1}(m)}=0$.
\end{lemma}

\begin{proof}
By symmetry, we may assume $i=1$.
Since the fiber $h_1^{-1}(m)$ is a Fano manifold with Picard number one, it is enough to check that $K_{X_1}$ is trivial on one curve in $h_1^{-1}(m)$.
Take a curve $C_2 \subset \psi^{-1}(m)$ in a $g_2|_{\psi^{-1}(m)}$-fiber and consider the push-forward $C_1 \coloneqq (g_{1}|_{\psi^{-1}(m)})_*(C_2)$.
Then
\[
K_{X_1} \cdot C_1 = g_1^* K_{X_1}\cdot C_2 = g_2^* K_{X_2} \cdot C_2 =0.
\]
Thus the assertion follows.
\end{proof}

The following completes the proof of Theorem~\ref{thm:simple}:

\begin{proposition}
The pair $({h_i^{-1}(m)},\sC_{Y_i/X_i}|_{h_i^{-1}(m)})$ is a simple Mukai pair.
\end{proposition}

\begin{proof}
It remains to check that $\sC_{Y_i/X_i}|_{h_i^{-1}(m)}$ is ample and $c_1(\sC_{Y_i/X_i}|_{h_i^{-1}(m)}) = c_1(h_i^{-1}(m))$.
The first assertion follows from Proposition~\ref{prop:canonicalbundle} and the second assertion follows from Lemma~\ref{lem:K=0} and  Proposition~\ref{prop:canonicalbundle}.
\end{proof}

\section{Construction of simple $K$-equivalence}\label{sect:const}
By Theorem~\ref{thm:simple}, simple $K$-equivalent maps are related to simple Mukai pairs: such a map can be seen as a family of simpler maps induced from simple Mukai pairs.
In this section, we discuss the inverse construction following \cite{Muk84,Nam03}, and explain how we can construct a simple $K$-equivalent map from a simple Mukai pair (or a family of simple Mukai pairs).

Let $X$ be a projective manifold and $Y \subset X$  a smooth subvariety of codimension at least two that satisfies the following conditions:
\begin{enumerate}
 \item $Y$ admits a smooth extremal contraction $h \colon Y \to M$.
 \item For each $h$-fiber $F$, the pair $(F,\sC_{Y/X}|_{F})$ is a simple Mukai pair.
\end{enumerate}

Denote by $\tX$ the blow-up of $X$ along $Y$, $E $ the exceptional divisor,  $g \colon E \to Y$ the natural projection and $\psi$ the composite $h \circ g$.
Note that $E$ is isomorphic to $ \bP(\sC_{Y/X})$ and the bundle $\cO_{\tX}(-E)|_E$ gives the relative tautological divisor of the projective bundle $ \bP(\sC_{Y/X})$.

\begin{lemma}
Let the notation be as above.
Then $E$ admits another $\bP^{r-1}$-bundle structure $g^{+}\colon E \to  Y^{+}$, with the following commutative diagram:
\[
 \xymatrix{
 & E \ar[ld]_{g} \ar[rd]^{g^+} & \\
 Y \ar[rd]_{h} & & Y^+ \ar[ld]^{h^+} \\
 & M. &
}
\]
Moreover, $\cO_{\tX}(-E)|_{E}$ gives a relative tautological divisor of $g^{+}$.
\end{lemma}

\begin{proof}
By our assumption, each fiber of the morphism  $\psi \colon E \to M$ is a Fano manifold with Picard number two which admits two $\bP^{r-1}$-bundle structures.
Thus $-K_{E}$ is $\psi$-ample.
Note that $\rho(E/M) =2$.
Therefore, by \cite[Theorem~3.25]{KM98} or \cite[Theorems 3-2-1, 4-2-1]{KMM87}, there exists the other extremal contraction $g^{+} \colon E \to Y^{+}$ with the following commutative diagram:
\[
\xymatrix{
 & E \ar[ld]_{g} \ar[rd]^{g^+} & \\
 Y \ar[rd]_{h} & & Y^+ \ar[ld]^{h^+} \\
 &M. &
}
\]
Take a point $m \in M$.
By \cite[Proposition~1.3]{Wis91b}, $g^{+}|_{\psi^{-1}(m)}$ is not finite, and hence the Stein factorization of $g^{+}|_{\psi^{-1}(m)}$ gives the other projective bundle structure.
In particular, $g^{+}$ is equidimensional.
Moreover, if $m \in M$ is general, then the morphism $g^{+}|_{\psi^{-1}(m)}$ gives the other projective bundle structure.
Thus general $g^{+}$-fiber is a projective space $\bP^{r-1}$.
Moreover, Proposition~\ref{prop:roof} shows that $\cO(-E)|_E$ restricts to $\cO(1)$ on general fibers $\bP^{r-1}$.
Therefore, by \cite[Lemma~2.12]{Fuj87}, $g^+$ is a $\bP^{r-1}$-bundle and $\cO(-E)|_E$ gives a relative tautological bundle of this projective bundle structure.
\end{proof}

\begin{proposition}[Construction of simple $K$-equivalent map]\label{prop:construction}
Let the notation be as above.
Then the blow-up $\tX$ admits another morphism $f^+ \colon \tX \to X^{+}$ onto a smooth complex manifold $X^{+}$ (may not be projective) with the following conditions:
\begin{enumerate}
 \item $X^+$ contains $Y^+$ as a closed subvariety.
 \item The morphism $f^{+} \colon \tX \to X^{+}$ restricts to $g^+$ on E.
 \item $f^+$ is the smooth blow-up along $Y^+ \subset X^+$.
 \item $f^* K_X = (f^+)^*K_{X_+}$.
\end{enumerate}
\end{proposition}

\begin{proof}
The first three conditions follow from \cite{Nak70,FN71}.
The last condition follows from adjunction.
\end{proof}

As is well-known, $X^+$ can be non-projective, and hence the map $\chi \colon X \dto X^+$ is, in general, not a map between \emph{projective} varieties.
The following lemma gives a sufficient condition for the projectivity of $X^+$ (cf.\ \cite[Proposition~1.3]{LLW10}).

\begin{proposition}\label{prop:flop}
In Proposition~\ref{prop:construction}, assume moreover that $X$ admits a birational contraction $\varphi \colon X \to Z$ to a projective variety $Z$ that satisfies the following conditions:
\begin{enumerate}
 \item The exceptional locus of $\varphi$ is $Y$.
 \item The Stein factorization of $\varphi |_Y$ gives the contraction $h \colon Y \to M$.
\end{enumerate}
Then the contraction $\varphi$ is a contraction of $K_X$-trivial ray, $X^+$ is projective and the map $X \dto X^+$ is the flop of $\varphi$.
\end{proposition}

\begin{proof}
By considering the push-forward via the inclusion map $Y \to X$, we have the half line $R_h \subset \cNE(X)$ corresponding to the extremal ray of $h\colon Y \to M$.

Arguing as in the proof of \cite[Chapter III, Theorem1.6]{Kol96}, we have an irreducible divisor $D \subset X$ such that $D \cdot R_h <0$.
Thus, for sufficiently small $\varepsilon > 0$, the pair $(X , \varepsilon D)$ is Kawamata log terminal, and $- K_X - \varepsilon D$ is $\varphi$-ample.
Thus, by \cite[Lemma~3-2-5]{KMM87}, the half line $R_h$ is actually an $(K_X + \varepsilon D)$-negative extremal ray of $\cNE(X)$ and the contraction $\varphi$ is associated to $R_h$.

The relative Picard number $\rho(\tX / Z)$ is two by \cite[Lemma~3-2-5]{KMM87} again, and, since $-K_{\tX}$ is $\psi$-ample, we have the other contraction of $\tX$ over $Z$ by the cone theorem.
This contraction is nothing but the morphism $f^+$.
Thus $X^+$ is projective, and the map $\chi \colon X \dto X^+$ is the flop.
\end{proof}

Finally, we construct a \emph{local model} of simple $K$-equivalence from a simple Mukai pair, or a family of simple Mukai pairs (see \cite[Section~1]{Nam03}).
\begin{proposition}
 Let $h \colon Y \to M$ be a smooth extremal contraction between smooth projective varieties $Y$ and $M$.
 Assume that there is a vector bundle $\sE$ on $Y$ such that, for each $h$-fiber $F$, the pair $(F,\sE|_{F})$ is a simple Mukai pair.
Then there exists a smooth projective variety $X$ that contains $Y$ as in the assumption of Proposition~\ref{prop:flop}.
\end{proposition}

\begin{proof}
Set $X \coloneqq \bP(\sE \oplus \cO)$.
Then the surjection $ \sE \oplus \cO \to \cO$ gives a section $Y'$ of the projection $\pi \colon X \to Y$.
We will denote by $h' \colon Y' \to M$ the composite $Y' \simeq Y \to M$, and by $\theta$ the composite $h \circ \pi$.

By construction, $\sC_{Y'/X} \simeq \sE$ via the identification $Y' \simeq Y$.
Thus it remains to show that $X$ admits a contraction $\varphi \colon X \to Z$ as in Proposition~\ref{prop:flop}.

Each fiber of $\theta$ is isomorphic to $\bP(\sE|_F \oplus \cO_F)$, where $F$ is a fiber of $h$.
Then, by using the definition of Mukai pairs, it is easy to check that $\bP(\sE|_F \oplus \cO_F)$ is a weak Fano variety, i.e.\ $-K_{\bP(\sE|_F \oplus \cO_F)}$ is nef and big.
Thus $-K_X$ is $\theta$-nef and $\theta$-big.
Thus, by the relative basepoint-free theorem \cite[Theorem~3.24]{KM98} or \cite[Theorem 3-1-1]{KMM87}, $-K_X$ defines a contraction $\varphi \colon  X \to Z$ over $M$:
\[
\xymatrix{
X \ar[r]^-{\varphi} \ar[d]_-{\pi}  \ar[rd]^-{\theta}& Z \ar[d] \\
Y \ar[r]_{h}& M
}
\]

Then the exceptional locus of $\varphi$ is $Y'$, and $\varphi |_{Y'}$ determines the contraction $Y' \to M$. This completes the proof.
\end{proof}

\section{Examples and classification}\label{sect:ex}
In this section, we firstly present examples of roofs of $\bP^{r-1}$-bundles and simple $K$-equivalent maps.
Secondly, we review the classification results of roofs.
Finally, by using the classification results and the structure theorem, we prove Theorem~\ref{thm:dim8}.

\subsection{Examples of roofs and simple $K$-equivalence}

\subsubsection{Homogeneous cases}
A \emph{rational homogeneous variety} is, by definition, a homogeneous variety of the form $G/P$, where $G$ is a semi-simple algebraic group and $P$ is a parabolic subgroup.
Such a variety is uniquely determined from its combinatoric data, called its \emph{marked Dynkin diagram}:
Let $G$ be a semi-simple group $G$ and $B$ a Borel subgroup of $G$.
Then we can attach a Dynkin diagram $D$ of a reduced root system by considering its Lie algebra.
Then there is a one-to-one correspondence between the set of parabolic subgroups contained in $B$ and the set of subsets $I \subset D$ (see for instance \cite[2.2]{MOSWW15}).
Our notation is compatible with \cite[2.2]{MOSWW15}.
Thus the correspondence is inclusion-reversing.
We will call the pair $(D,I)$ the \emph{marked Dynkin diagram} for the homogeneous variety $G/P$.

Fix a semi-simple group $G$, and denote by $D$ its Dynkin diagram.
Then, by the above correspondence, we have the parabolic subgroup $P(I)$ for each subset $I \subset D$.
It is known that the Picard number of a rational homogeneous manifold $G/P(I)$ is $\# I$.
Also, by construction, if $I \subset J$, then we have the contraction $G/P(J) \to G/P(I) $, whose fibers are the rational homogeneous manifold corresponding to the marked Dynkin diagram $(D \setminus I, J \setminus I)$ (here the Dynkin diagram $D\setminus I$ is obtained by removing the nodes in $I$ and the edges touching the nodes in $I$).
In particular, a subset of $D$ with one element gives a maximal parabolic subgroup, and hence it gives a rational homogeneous variety with Picard number one.
For such varieties, its dimension and index are determined from the combinatoric data (see e.g.\ \cite[Corollary~2.4]{Sno93}).
Thus, by combining with the Kobayashi-Ochiai theorem \cite{KO73}, we can see that  a rational homogeneous variety $G/P$ is isomorphic to a projective space $\bP^{r-1}$, if and only if its marked Dynkin diagram is the following two types:
\begin{enumerate}
\item
$\dPA$

\item
$\dPC$
\end{enumerate}
Here the marking is specified by the black circle.

Let $W$ be a rational homogeneous manifold $W = G/P$ and $(D,I)$ be its marked Dynkin diagram.
Assume that $W$ is a roof of $\bP^{r-1}$-bundles.
Then $I$ consists of two elements $i$ and $j$.
Since it admits two $\bP^{r-1}$-bundle structures, the marked Dynkin diagrams $(D \setminus \{i\}, \{j\})$ and $(D \setminus \{j\}, \{i\})$ are one of the two marked Dynkin diagrams as above.
Conversely, if we are given a marked Dynkin diagram $(D,\{i,j\})$ as above, then the corresponding rational homogeneous variety is a roof of $\bP^{r-1}$-bundles.
Thus, by checking for each cases, we have the following seven examples of homogeneous roofs, and hence seven examples of simple $K$-equivalent maps.

\begin{example}[Type $A_{r-1}\times A_{r-1}$]
Set $W \coloneqq \bP^{r-1} \times \bP^{r-1}$.
Then $W$ is a roof of $\bP^{r-1}$-bundles.
Note that the variety $W$ is  a homogeneous variety whose automorphism group is a semi-simple group of type $A_{r-1} \times A_{r-1}$, and it corresponds to the following marked Dynkin diagram.
\[
\dAA
\]

We will call this variety a \emph{roof of type $A_{r-1} \times A_{r-1}$}.
A simple $K$-equivalent map is called \emph{type $A_{r-1} \times A_{r-1}$}, if each $\psi$-fiber in the diagram \eqref{diagram:str} is isomorphic to the roof of type $A_{r-1} \times A_{r-1}$.
Note that simple $K$-equivalent maps of type $A_{r-1} \times A_{r-1}$ are nothing but so-called standard flops (see Remark~\ref{rem:fibration}).
\end{example}

\begin{example}[Type $A_{r}^M$]\label{ex:Mukai}
Consider the flag variety $W \coloneqq \Fl (1,r;r+1)$, which parametrizes the flags of subspaces $(V_1\subset V_r)$ with $\dim V_i=i$ in a vector space $\bC^{r+1}$.
Then $W$ is a roof of $\bP^{r-1}$-bundles.

The roof $W$ admits two natural projections $\pr_1 \colon \Fl (1,r;r+1) \to \Gr (1,r+1)$ and $\pr_2 \colon \Fl (1,r;r+1) \to \Gr (r,r+1)$.
The fibers of these projections are isomorphic to $\bP^{r-1}$.

$W$ is a homogeneous variety whose automorphism group is a semi-simple group of type $A_r$.
The marked Dynkin diagram of $W$ is the following.
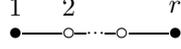
\[
\dAM
\]
We will call this variety a \emph{roof of type $A_{r}^M$}.
A simple $K$-equivalent map is called \emph{type $A_{r}^{M}$}, if each $\psi$-fiber in the diagram \eqref{diagram:str} is isomorphic to the roof of type $A_{r}^{M}$.
Note that this flag variety $W$ is isomorphic to the projectivized tangent bundle of a projective space $\bP(T_{\bP^r})$, and hence simple $K$-equivalent maps of type $A_{r}^M$ are so-called Mukai flops (see Remark~\ref{rem:fibration}).
\end{example}

\begin{example}[Type $A_{2r-2}^G$ ($r \geq 3$)]\label{ex:AG}
Consider the flag variety $W \coloneqq \Fl (r-1,r;2r-1)$.
Then, similarly to Example~\ref{ex:Mukai}, $W$ is a roof of $\bP^{r-1}$-bundles.
The images of projections are the Grassmannian varieties $\Gr(r-1;2r-1)$ and $\Gr(r;2r-1)$ respectively.
$W$ is a rational homogeneous variety whose marked Dynkin diagram is the following.
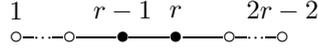
\[
\dAG
\]
We call this variety a \emph{roof of type $A_{2r-2}^G$}.
A simple $K$-equivalent map is called \emph{type $A_{2r-2}^G$}, if each $\psi$-fiber in the diagram \eqref{diagram:str} is isomorphic to the roof of type $A_{2r-2}^G$.
\end{example}

\begin{example}[Type $C_{\frac{3r}{2}-1}$ ($r$ even)]\label{ex:C}
Let $r \geq 2$ be an even integer and
fix a symplectic bilinear form on a vector space $\bC^{3r-2}$.
Consider the symplectic flag variety $\SFl (r-1,r;3r-2)$, which parametrizes the flags of isotropic subspaces $(V_{r-1} \subset V_r) $ with $\dim V_i=i$.
Then $W$ is a roof of $\bP^{r-1}$-bundles.
The images of projections are the symplectic Grassmannians $\SG(r-1;3r-2)$ and $\SG(r;3r-2)$ respectively.
$W$ is a rational homogeneous variety whose marked Dynkin diagram is the following.
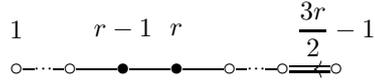
\[
\dC
\]
We will call this variety a \emph{roof of type $C_{\frac{3r}{2}-1}$}.
A simple $K$-equivalent map is called \emph{type $C_{\frac{3r}{2}-1}$}, if each $\psi$-fiber in the diagram \eqref{diagram:str} is isomorphic to the roof of type $C_{\frac{3r}{2}-1}$.
Note that Abuaf's flop in \cite{Seg16} is a $K$-equivalent map of type $C_2$.
\end{example}

\begin{example}[Type $D_{r}$ ($r\geq4$)]\label{ex:D}
Fix a non-degenerate quadratic form on a vector space $\bC^{2r}$ and consider the orthogonal Grassmann variety $\OG (r-1;2r)$, which parametrizes the $r-1$-dimensional isotropic subspaces.
Then $W$ is a rational homogeneous variety whose marked Dynkin diagram is the following.
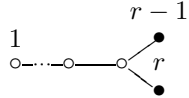
\[
\dD
\]
Thus $W$ gives a roof of $\bP^{r-1}$-bundles.
The images of projections are the orthogonal Grassmannians $\OG^+(r;2r)$ and $\OG^-(r;2r)$, which are the connected components of the orthogonal Grassmannian $\OG(r;2r)$.
We will call this variety a \emph{roof of type $D_{r}$}.
A simple $K$-equivalent map is called \emph{type $D_r$}, if each $\psi$-fiber in the diagram \eqref{diagram:str} is isomorphic to the roof of type $D_{r}$.
\end{example}

\begin{example}[Type $F_4$ ($r=3$)]\label{ex:F}
Consider a rational homogeneous variety $W$ whose marked Dynkin diagram is the following.
\[
\dF
\]
Then $W$ is a roof of $\bP^{2}$-bundles.
We will call this variety a \emph{roof of type $F_4$}.
A simple $K$-equivalent map is called \emph{type $F_4$}, if each $\psi$-fiber in the diagram \eqref{diagram:str} is isomorphic to the roof of type $F_4$.
\end{example}

\begin{example}[Type $G_2$ ($r=2$)]\label{ex:G}
Consider a rational homogeneous variety $W$ whose marked Dynkin diagram is the following:
\[
\dG
\]
Then $W$ gives a roof of $\bP^{1}$-bundles.
We will call this variety a \emph{roof of type $G_2$}.
A simple $K$-equivalent map is called \emph{type $G_2$}, if each $\psi$-fiber in the diagram \eqref{diagram:str} is isomorphic to the roof of type $G_2$.
Note that the flop studied in \cite{Ued18} is of type $G_2$.
\end{example}

\begin{remark}\label{rem:fibration}
The definition of standard flops (type $A_{r-1}\times A_{r-1}$) and Mukai flops (type $A_{r}^M$) are slightly different from the definition in \cite{Li18};
In our definition, we do \emph{not} assume that the morphisms $h_i$ are projective bundles, i.e. it comes from the projectivization of a vector bundle.
In fact, there are simple $K$-equivalent maps of these types, where the morphisms $h_i$ are not projective bundles (see the following example).
\end{remark}

\begin{example}
Consider a smooth $\bP^{r-1}$-fibration $h \colon Y \to M$, which is not a $\bP^{r-1}$-bundle (note that such an example exists already in dimension $3$ over a surface $M$, see e.g.\ \cite{BOSS96}).
Then by letting $\sE \coloneqq \cO^{\oplus r}$ or $T_{h}$ (the relative tangent bundle) and applying the construction in Section~\ref{sect:const}, we obtain a simple $K$-equivalent map of type $A_{r-1}\times A_{r-1}$ or $A_r^{M}$.
For this example, the flopping locus is isomorphic to $h \colon Y \to M $, which is not a $\bP^{r-1}$-bundle.
\end{example}

\subsection{Non-homogeneous roof}
Here we will provide one example of roof, which is not homogeneous, based on \cite{Ott88,Ott90,Kan16Ext}.

Let $\bQ^5$ be a smooth $5$-dimensional hyperquadric.
Then the Chow group $A_i(\bQ^5)$ is isomorphic to the group $\bZ$ for all $i \in \{0,1,2,3,4,5\}$.
We will identify an element  $A_i(\bQ^5)$ with an integer.
\begin{definition}A vector bundle $\sG$ of rank $3$ on $\bQ^5$ is called an \emph{Ottaviani bundle} if it is stable and $(c_1(\sG),c_2(\sG),c_3(\sG))=(2,2,2)$.
\end{definition}
Such a bundle is constructed and studied in \cite{Ott88}.
Herein, we include one description of the projectivization of this bundle.
See \cite{Ott88,Ott90,Kan16Ext,Kan17Mukai} for other properties of this bundle and several characterizations.

In \cite[Section~2]{Kan16Ext}, it is proved that the projectivization $\bP(\sG)$ is a roof of $\bP^2$-bundles and $\bP(\sG)$ is isomorphic to the following manifold (cf.\ \cite{Ott90}):

\begin{example}[Type $G_2^\dagger$ ($r=3$)]\label{ex:G*}
Let $\bO$ be the Cayley octonions, and denote by $\blank \cdot \blank$ be its Cayley product.
Let $W$ be a closed submanifold of $\bP(\im \bO) \times \bP(\im \bO)$ defined as follows:
\[
\{ (x,y) \in \bP(\im \bO) \times \bP(\im \bO) \mid x\cdot x = x\cdot y = y\cdot y = 0 \}.
\]
Then, the image of each projection $\pr_i|_{W}$ is isomorphic to a smooth hyperquadric $\bQ^5$ in $\bP(\im \bO) \simeq \bP^6$.
Moreover the projection $\pr_i|_{W} \colon W \to \bQ^5$ is a $\bP^2$-bundle, and these define the structure of a roof on $W$.
Note that the automorphism group of $\bO$ is a semi-simple group of type $G_2$, and, by the construction, $W$ admits the action of a semi-simple group of type $G_2$.
We will call this variety $W$ a roof of type $G_2^\dagger$.
A simple $K$-equivalent map is said to be of type $G_2^\dagger$, if each $\psi$-fiber in the diagram \eqref{diagram:str} is isomorphic to the roof of type $G_2^\dagger$.
Note that this variety $W$ is not homogeneous (see e.g.\ \cite[Theorem~2.2]{Kan16Ext}).
\end{example}

\subsubsection{List of roofs}
Let $W$ be a roof of $\bP^{r-1}$-bundles.
Then we have the following diagram with two $\bP^{r-1}$-bundle structures:
\[
\xymatrix{
& W \ar[ld]_-{p_1} \ar[rd]^-{p_2} & \\
V_1 & & V_2.
}
\]
So far we have constructed eight examples of roofs $A_{r-1} \times A_{r-1}$, $A_r^M$, $A_{2r-2}^G$, $C_{\frac{3r}{2}-1}$, $D_r$, $F_4$, $G_2$ and $G_2^{\dagger}$.
The following is the list of these examples.
The second column lists the marked Dynkin diagrams for homogeneous roofs, and the last column lists the triples $(\dim V_i, r_{V_1},r_{V_2})$, where $r_{V_i}$ is the index of $V_i$:

\begin{table}[!ht]
\begin{tabular}{lcc}
Type & Marked Dynkin diagram & $(\dim V_i, r_{V_1},r_{V_2})$ \\
$A_{r-1}\times A_{r-1}$ & \dAA & $(r-1,r,r)$ \\
$A_r^M$ & \dAM & $(r,r+1,r+1)$ \\
$A_{2r-2}^G$ & \dAG & $(r(r-1),2r-1,2r-1)$ \\
$C_{\frac{3r}{2}-1}$ ($r$ even) & \dC & $(\frac{3r(r-1)}{2},2r,2r-1)$ \\
$D_r$ & $\begin{gathered} \dD \end{gathered}$ & $(\frac{r(r-1)}{2},2r-2,2r-2)$ \\
$F_4$ ($r=3$) & \dF & $(20,5,7)$ \\
$G_2$ ($r=2$) & \dG & $(5,3,5)$ \\
$G_2^\dagger$ ($r=3$) & & $(5,5,5)$ \\
\end{tabular}
\end{table}

\subsection{Classification results of roofs}

By combining classification results of roofs, we have the following:

\begin{theorem}\label{thm:class}
Let $W$ be a roof of $\bP^{r-1}$-bundes with dimension $n+r-1$.
Then $W$ is isomorphic to one of the above examples, if one of the following holds:
\begin{enumerate}
 \item \label{thm:class1} $r \geq n-2$.
 \item \label{thm:class2} $r = 2$.
 \item \label{thm:class3} $\dim W \leq 7$.
\end{enumerate}
More precisely, the following hold:
\begin{enumerate}
\item If \ref{thm:class1} holds, then $W$ is of type $A_{r-1} \times A_{r-1}$, $A_r^M$,  $C_{2}$, $D_4$  or $G_2^{\dagger}$.
\item If \ref{thm:class2} holds, then $W$ is of type $A_{1} \times A_{1}$, $A_2^M$,  $C_{2}$,  $G_2$.
\item If \ref{thm:class3} holds, then $W$ is of type $A_{r-1} \times A_{r-1}$ ($r \leq 3$), $A_r^M$ ($r \leq 3$),  $C_{2}$, $G_2$ or $G_2^{\dagger}$.
\end{enumerate}
\end{theorem}

\begin{proof}
In the first case, the assertion follows from the classification of Mukai pairs with large rank \cite{Fuj92,Pet90,Pet91,YZ90,Wis89b,PSW92b,NO07,Kan17Mukai,Kan18Mukai} (cf.\ \cite{Occ05}).
In the second case, the assertion follows from the classification of Fano manifolds with Picard rank two whose extremal contractions are $\bP^1$-bundles \cite{MOS14a,Wat14b}.
The last assertion is a consequence of \ref{thm:class1} and \ref{thm:class2}.
\end{proof}

As a corollary of the above classification and the structure theorem, we have the following:
\begin{corollary}\label{cor:class}
 Let $\chi \colon X_1 \dto X_2$ be a simple $K$-equivalent map in codimension $r$ as in \eqref{diagram:str}.
Then the following hold:
\begin{enumerate}
\item If $r \geq \dim Y_i - \dim M-2$, then $\chi$ is of type $A_{r-1} \times A_{r-1}$, $A_r^M$,  $C_{2}$, $D_4$  or $G_2^{\dagger}$.
\item If $r \leq 2$, then $\chi$ is of type $A_{1} \times A_{1}$, $A_2^M$, $C_{2}$ or  $G_2$.
\item If $\dim X_i \leq 8$, then $\chi$ is of type $A_{r-1} \times A_{r-1}$ ($r \leq 3$), $A_r^M$ ($r \leq 3$), $C_{2}$, $G_2$ or $G_2^{\dagger}$.
\end{enumerate}
\end{corollary}
\begin{proof}
This follows from Theorem~\ref{thm:simple} and Theorem~\ref{thm:class}.
\end{proof}

\begin{remark}
 In Theorem~\ref{thm:class} and Corollary~\ref{cor:class}, we have shown that, in several cases, simple $K$-equivalent maps or roofs are one of the examples constructed above.
To the best of the author's knowledge, these are the all known examples of roofs, and hence of simple $K$-equivalence.
\end{remark}

\section{Symplectic varieties}\label{sect:symp}
Mukai flops (or simple $K$-equivalent maps of type $A_r^M$ in our terminology) are introduced by Mukai in the context of the geometry of symplectic varieties \cite{Muk84}.
I learned from Duo Li that the following question is raised by Daniel Huybrechts:

\begin{question}
If $\chi \colon X \dto X^+$ is a simple $K$-equivalent map between symplectic varieties $X$, then is $\chi$ a Mukai flop?
\end{question}

In his paper \cite[Theorem~1.7]{Li18}, Li obtained a positive answer for this question if the Picard rank of the center of the birational map is one.
The following theorem is obtained via the discussion with Duo Li, which answers positively the above question:

\begin{theorem}\label{thm:symp}
Let $X$ be a projective symplectic manifold of dimension $2n$, i.e.\ a smooth projective variety that admits a symplectic form $\omega \in H^0(\Omega_{X}^2)$, and $\chi \colon X \dto X^+$ a simple $K$-equivalent map.
Then $\chi$ is a Mukai flop, or equivalently a simple $K$-equivalent map of type $A_r^M$.
\end{theorem}

\begin{proof}
We will use a similar notation as in Theorem~\ref{thm:simple}.
Then, by arguing as in the proof of \cite[Proposition~3.1]{Muk84}, we see that $r = \codim _X Y \geq \dim F = \dim Y - \dim M$ and $F$ is isotropic, where $F$ is a fiber of $h$.
Then, by Corollary~\ref{cor:class}, $\chi$ is of type $A_{r-1}\times A_{r-1}$ or $A_r^M$.

In any case, $F$ is isomorphic to a projective space $\bP^{\dim F}$.
Since $F$ is isotropic, the tangent bundle $T_{F}$ is a subbundle of $\sC_{F/X}$.
Note that  $\sC_{Y/X}|_{F}$ is isomorphic to $\cO_F(1)^{\oplus\dim F +1}$ or $T_F$, and there is the following exact sequence:
\[
0 \to \sC_{Y/X}|_{F} \to \sC_{F/X} \to \sC_{F/Y} \simeq \cO_{F}^{\oplus \dim Y -\dim M} \to 0.
\]
Since there are no non-trivial morphisms from $T_F$ to $\sC_{F/Y} \simeq \cO_{F}^{\oplus \dim Y -\dim M}$, the subbundle $T_{F} \subset \sC_{F/X}$ is contained in $\sC_{Y/X}|_{F}$.
Also there are no non-trivial morphisms from $T_F$ to $\cO_F(1)^{\oplus\dim F +1}$.
Thus $\sC_{Y/X}|_{F} \simeq T_F$.
This completes the proof.
\end{proof}

\bibliographystyle{amsalpha}
\bibliography{}

\newcommand{\etalchar}[1]{$^{#1}$}
\providecommand{\bysame}{\leavevmode\hbox to3em{\hrulefill}\thinspace}
\providecommand{\MR}{\relax\ifhmode\unskip\space\fi MR }
\providecommand{\MRhref}[2]{%
  \href{http://www.ams.org/mathscinet-getitem?mr=#1}{#2}
}
\providecommand{\href}[2]{#2}
\begin{thebibliography}{MOSC{\etalchar{+}}15}

\bibitem[AW01]{AW01}
Marco Andreatta and Jaros{\l}aw~A. Wi{\'s}niewski, \emph{On manifolds whose
  tangent bundle contains an ample subbundle}, Invent. Math. \textbf{146}
  (2001), no.~1, 209--217. \MR{1859022 (2002h:14070)}

\bibitem[BO95]{BO95}
A.~{Bondal} and D.~{Orlov}, \emph{Semiorthogonal decomposition for algebraic
  varieties}, arXiv:alg-geom/9506012v1, 1995.

\bibitem[BOSS96]{BOSS96}
R.~Braun, G.~Ottaviani, M.~Schneider, and F.-O. Schreyer, \emph{Classification
  of conic bundles in {${\bf P}_5$}}, Ann. Scuola Norm. Sup. Pisa Cl. Sci. (4)
  \textbf{23} (1996), no.~1, 69--97. \MR{1401418}

\bibitem[Cam81]{Cam81}
F.~Campana, \emph{Cor\'{e}duction alg\'{e}brique d'un espace analytique
  faiblement k\"{a}hl\'{e}rien compact}, Invent. Math. \textbf{63} (1981),
  no.~2, 187--223. \MR{610537}

\bibitem[Deb01]{Deb01}
Olivier Debarre, \emph{Higher-dimensional algebraic geometry}, Universitext,
  Springer-Verlag, New York, 2001. \MR{1841091 (2002g:14001)}

\bibitem[DPS94]{DPS94}
Jean-Pierre Demailly, Thomas Peternell, and Michael Schneider, \emph{Compact
  complex manifolds with numerically effective tangent bundles}, J. Algebraic
  Geom. \textbf{3} (1994), no.~2, 295--345. \MR{1257325 (95f:32037)}

\bibitem[FN72]{FN71}
Akira Fujiki and Shigeo Nakano, \emph{Supplement to ``{O}n the inverse of
  monoidal transformation''}, Publ. Res. Inst. Math. Sci. \textbf{7} (1971/72),
  637--644. \MR{0294712}

\bibitem[Fuj87]{Fuj87}
Takao Fujita, \emph{On polarized manifolds whose adjoint bundles are not
  semipositive}, Algebraic geometry, {S}endai, 1985, Adv. Stud. Pure Math.,
  vol.~10, North-Holland, Amsterdam, 1987, pp.~167--178. \MR{946238}

\bibitem[Fuj92]{Fuj92}
\bysame, \emph{On adjoint bundles of ample vector bundles}, Complex algebraic
  varieties ({B}ayreuth, 1990), Lecture Notes in Math., vol. 1507, Springer,
  Berlin, 1992, pp.~105--112. \MR{1178722}

\bibitem[{Har}17]{Har17}
Wahei {Hara}, \emph{{On derived equivalence for Abuaf flop: mutation of
  non-commutative crepant resolutions and spherical twists}}, arXiv:
  1706.04417v2, 2017.

\bibitem[Kan16]{Kan16Ext}
Akihiro Kanemitsu, \emph{Extremal rays and nefness of tangent bundles},
  arXiv:1605.04680v1, to appear in Michigan Math. J., 2016.

\bibitem[Kan17a]{Kan17Mukai}
\bysame, \emph{Classification of {M}ukai pairs with corank $3$},
  arXiv:1704.04995v2, to appear in Ann. Inst. Fourier (Grenoble), 2017.

\bibitem[Kan17b]{Kan17}
\bysame, \emph{Fano 5-folds with nef tangent bundles}, Math. Res. Lett.
  \textbf{24} (2017), no.~5, 1453--1475. \MR{3747171}

\bibitem[Kan18]{Kan18Mukai}
Akihiro. Kanemitsu, \emph{Classification of {M}ukai pairs with dimension $4$
  and rank $2$}, arXiv:1806.07587v1, 2018.

\bibitem[Kaw02]{Kaw02}
Yujiro Kawamata, \emph{{$D$}-equivalence and {$K$}-equivalence}, J.
  Differential Geom. \textbf{61} (2002), no.~1, 147--171. \MR{1949787}

\bibitem[KM98]{KM98}
J{\'a}nos Koll{\'a}r and Shigefumi Mori, \emph{Birational geometry of algebraic
  varieties}, Cambridge Tracts in Mathematics, vol. 134, Cambridge University
  Press, Cambridge, 1998, With the collaboration of C. H. Clemens and A. Corti,
  Translated from the 1998 Japanese original. \MR{1658959 (2000b:14018)}

\bibitem[KMM87]{KMM87}
Yujiro Kawamata, Katsumi Matsuda, and Kenji Matsuki, \emph{Introduction to the
  minimal model problem}, Algebraic geometry, {S}endai, 1985, Adv. Stud. Pure
  Math., vol.~10, North-Holland, Amsterdam, 1987, pp.~283--360. \MR{946243
  (89e:14015)}

\bibitem[KMM92]{KMM92a}
J{\'a}nos Koll{\'a}r, Yoichi Miyaoka, and Shigefumi Mori, \emph{Rational
  connectedness and boundedness of {F}ano manifolds}, J. Differential Geom.
  \textbf{36} (1992), no.~3, 765--779. \MR{1189503 (94g:14021)}

\bibitem[KO73]{KO73}
Shoshichi Kobayashi and Takushiro Ochiai, \emph{Characterizations of complex
  projective spaces and hyperquadrics}, J. Math. Kyoto Univ. \textbf{13}
  (1973), 31--47. \MR{0316745}

\bibitem[Kol96]{Kol96}
J{\'a}nos Koll{\'a}r, \emph{Rational curves on algebraic varieties}, Ergebnisse
  der Mathematik und ihrer Grenzgebiete. 3. Folge. A Series of Modern Surveys
  in Mathematics [Results in Mathematics and Related Areas. 3rd Series. A
  Series of Modern Surveys in Mathematics], vol.~32, Springer-Verlag, Berlin,
  1996. \MR{1440180 (98c:14001)}

\bibitem[Li18]{Li18}
Duo Li, \emph{On certain $k$-equivalent birational maps}, Math. Z.
  https://doi.org/10.1007/s00209-018-2169-z, 2018.

\bibitem[LLW10]{LLW10}
Yuan-Pin Lee, Hui-Wen Lin, and Chin-Lung Wang, \emph{Flops, motives, and
  invariance of quantum rings}, Ann. of Math. (2) \textbf{172} (2010), no.~1,
  243--290. \MR{2680420}

\bibitem[MOSC14]{MOS14a}
Roberto Mu{\~n}oz, Gianluca Occhetta, and Luis~Eduardo Sol{\'a}~Conde, \emph{On
  rank 2 vector bundles on {F}ano manifolds}, Kyoto J. Math. \textbf{54}
  (2014), no.~1, 167--197. \MR{3178550}

\bibitem[MOSC{\etalchar{+}}15]{MOSWW15}
Roberto Mu{\~n}oz, Gianluca Occhetta, Luis~E. Sol{\'a}~Conde, Kiwamu Watanabe,
  and Jaros{\l}aw~A. Wi{\'s}niewski, \emph{A survey on the
  {C}ampana-{P}eternell conjecture}, Rend. Istit. Mat. Univ. Trieste
  \textbf{47} (2015), 127--185. \MR{3456582}

\bibitem[Muk84]{Muk84}
Shigeru Mukai, \emph{Symplectic structure of the moduli space of sheaves on an
  abelian or {$K3$} surface}, Invent. Math. \textbf{77} (1984), no.~1,
  101--116. \MR{751133}

\bibitem[Muk88]{Muk88}
\bysame, \emph{Problems on characterization of the complex projective space},
  Birational Geometry of Algebraic Varieties, Open Problems, Katata, the 23rd
  Int'l Symp., Taniguchi Foundation, 1988, pp.~57--60.

\bibitem[Nak71]{Nak70}
Shigeo Nakano, \emph{On the inverse of monoidal transformation}, Publ. Res.
  Inst. Math. Sci. \textbf{6} (1970/71), 483--502. \MR{0294710}

\bibitem[Nam03]{Nam03}
Yoshinori Namikawa, \emph{Mukai flops and derived categories}, J. Reine Angew.
  Math. \textbf{560} (2003), 65--76. \MR{1992802}

\bibitem[NO07]{NO07}
Carla Novelli and Gianluca Occhetta, \emph{Ruled {F}ano fivefolds of index
  two}, Indiana Univ. Math. J. \textbf{56} (2007), no.~1, 207--241. \MR{2305935
  (2008e:14060)}

\bibitem[Occ05]{Occ05}
Gianluca Occhetta, \emph{A note on the classification of {F}ano manifolds of
  middle index}, Manuscripta Math. \textbf{117} (2005), no.~1, 43--49.
  \MR{2142900}

\bibitem[Ott88]{Ott88}
Giorgio Ottaviani, \emph{Spinor bundles on quadrics}, Trans. Amer. Math. Soc.
  \textbf{307} (1988), no.~1, 301--316. \MR{936818}

\bibitem[Ott90]{Ott90}
\bysame, \emph{On {C}ayley bundles on the five-dimensional quadric}, Boll. Un.
  Mat. Ital. A (7) \textbf{4} (1990), no.~1, 87--100. \MR{1047517}

\bibitem[Pet90]{Pet90}
Thomas Peternell, \emph{A characterization of {${\bf P}\sb n$} by vector
  bundles}, Math. Z. \textbf{205} (1990), no.~3, 487--490. \MR{1082869}

\bibitem[Pet91]{Pet91}
\bysame, \emph{Ample vector bundles on {F}ano manifolds}, Internat. J. Math.
  \textbf{2} (1991), no.~3, 311--322. \MR{1104121}

\bibitem[PSW92]{PSW92b}
Thomas Peternell, Micha{\l} Szurek, and Jaros{\l}aw~A. Wi{\'s}niewski,
  \emph{Fano manifolds and vector bundles}, Math. Ann. \textbf{294} (1992),
  no.~1, 151--165. \MR{1180456}

\bibitem[SCW04]{SW04}
Luis~Eduardo Sol{\'a}~Conde and Jaros{\l}aw~A. Wi{\'s}niewski, \emph{On
  manifolds whose tangent bundle is big and 1-ample}, Proc. London Math. Soc.
  (3) \textbf{89} (2004), no.~2, 273--290. \MR{2078708 (2005c:14018)}

\bibitem[Seg16]{Seg16}
Ed~Segal, \emph{A new 5-fold flop and derived equivalence}, Bull. Lond. Math.
  Soc. \textbf{48} (2016), no.~3, 533--538. \MR{3509912}

\bibitem[Sno93]{Sno93}
Dennis~M. Snow, \emph{The nef value and defect of homogeneous line bundles},
  Trans. Amer. Math. Soc. \textbf{340} (1993), no.~1, 227--241. \MR{1144015}

\bibitem[Ued18]{Ued18}
Kazushi Ueda, \emph{{$G_2$}-grassmannians and derived equivalences},
  manuscripta math. https://doi.org/10.1007/s00229-018-1090-4, 2018.

\bibitem[Wat14]{Wat14b}
Kiwamu Watanabe, \emph{{$\mathbb{P}^1$}-bundles admitting another smooth
  morphism of relative dimension one}, J. Algebra \textbf{414} (2014),
  105--119. \MR{3223392}

\bibitem[Wi{\'s}89]{Wis89b}
Jaros{\l}aw~A. Wi{\'s}niewski, \emph{Ruled {F}ano {$4$}-folds of index {$2$}},
  Proc. Amer. Math. Soc. \textbf{105} (1989), no.~1, 55--61. \MR{929433}

\bibitem[Wi{\'s}91]{Wis91b}
\bysame, \emph{On deformation of nef values}, Duke Math. J. \textbf{64} (1991),
  no.~2, 325--332. \MR{1136378}

\bibitem[YZ90]{YZ90}
Yun-Gang Ye and Qi~Zhang, \emph{On ample vector bundles whose adjunction
  bundles are not numerically effective}, Duke Math. J. \textbf{60} (1990),
  no.~3, 671--687. \MR{1054530}

\end{thebibliography}
\end{document}